\documentclass[dvipsnames, 11pt]{article}

\usepackage[latin1, utf8]{inputenc}
\usepackage[T1]{fontenc}
\usepackage[english]{babel}
\usepackage{todonotes}
\usepackage{tabularx}
\usepackage{makecell}
\usepackage{multirow}
\usepackage{pifont} 
\usepackage{hyperref}
\usepackage{amsmath}
\usepackage{amssymb}
\usepackage{amsthm}
\usepackage{authblk}
\usepackage{mathtools}
\usepackage{xkeyval}
\usepackage{algorithm}
\usepackage{algpseudocode}
\usepackage{stmaryrd}
\usepackage{float}

\usepackage[left=2.5cm,top=3cm,right=2.5cm,bottom=3cm,bindingoffset=0.5cm]{geometry}
\newcommand{\R}{\mathbb{R}}
\newcommand{\N}{\mathbb{N}}

\newcommand{\C}{\mathbb{C}}

\newcommand{\U}{\mathcal{U}}

\DeclareMathOperator*{\argmax}{arg\,max}

\newcommand {\e}  {\varepsilon}
\newcommand{\vertiii}[1]{{\vert\kern-0.25ex\vert\kern-0.25ex\vert #1 
    \vert\kern-0.25ex\vert\kern-0.25ex\vert}}
    
\usepackage{subcaption}
    
\newtheorem{prpstn}{Proposition}[section]
\newtheorem{lmm}{Lemma}[section]
\newtheorem{thrm}{Theorem}[section]
\newtheorem{dfntn}{Definition}[section]
\newtheorem{crllr}{Corollary}[section]
\newtheorem{rmrk}{Remark}[section]

\everymath{\displaystyle}
\title{Random Leja points}
\author[a]{Camille Pouchol}
\affil[a]{Laboratoire MAP5, UMR 8145,
Université Paris Cité, 75006 Paris, France.\footnote{Email address: camille.pouchol@u-paris.fr.}}
\date{\empty}

\begin{document}
\maketitle

\begin{abstract}
Leja points on a compact $K \subset \C$ are known to provide efficient points for interpolation, but their actual implementation can be computationally challenging. So-called pseudo Leja points are a more tractable solution, yet they require a tailored implementation to the compact at hand.  
We introduce several more flexible random alternatives, starting from a new family we call random Leja points. To make them tractable, we propose an approximate version which relies on the Metropolis-Hastings algorithm with the uniform measure. We also analyse a different family of points inspired by recently introduced randomised admissible meshes, obtained by uniform sampling. 
When the number of iterations or drawn points is appropriately chosen, we establish that the two resulting families of points provide good points for interpolation. That is, they almost surely lead to convergent interpolating polynomials for holomorphic functions. The two last families of points are readily implemented assuming one knows how to sample uniformly at random in $K$. These makes them more modular than competing deterministic methods.
We run numerical experiments to compare the proposed methods in terms of accuracy and computational complexity, for various types of compact sets.
\end{abstract}

\section{Introduction}

\subsection{From Leja to random Leja points}
\paragraph{Leja points.} For a given compact set $K \subset \C$, a set of \textit{Leja points} for $K$ refers to any sequence $(z_n) \in K^\N$ satisfying 
\begin{equation}
\label{def_det_leja}
z_{n} \in \argmax_{z \in K} \prod_{i=0}^{n-1} |z-z_i|.
\end{equation}
Introducing the short-hand notation $\pi_n(z) = \textstyle  \prod_{i=0}^{n-1} (z-z_i)$ and the supremum norm $\|\cdot\|_K$ on $K$, this condition rewrites compactly as $|\pi_n(z_n)| = \|\pi_n\|_K$.
These points have been introduced by Leja in~\cite{Leja1957}, and are most notably used in potential theory and in polynomial interpolation; the latter field is our main motivation.

Various theoretical results and numerical experiments have shown that Leja points enjoy good properties when it comes to polynomial interpolation. Yet, the very definition~\eqref{def_det_leja} poses a challenge from the computational point of view. Indeed, finding the next point requires solving a non-convex optimisation problem with many local maxima. 

\paragraph{Pseudo Leja points.}
The seminal work~\cite{PseudoLeja2012} introduces \textit{pseudo Leja points}, where the maximisation condition~\eqref{def_det_leja} is relaxed as follows.
 \begin{dfntn}
 \label{pseudoLeja}
 A set of pseudo-Leja points for $K$ is any sequence $(z_n) \in K^\N$ such that
 \[|\pi_n(z_n)| \geq \tau_n \|\pi_n\|_K,\] where $0 < \tau_n \leq 1$ is subexponential, i.e., satisfies $\tau_n^{1/n} \to 1$ as $n \to +\infty$. 
 
If $\tau_n \sim n^{-\beta}$\footnote{We use the notation $u_n \sim v_n$ to mean that both $u_n = O(v_n)$ and $v_n = O(u_n)$ hold.}, we shall say that the points $z_n$ are pseudo-Leja points of order $\beta$ for $K$. 
\end{dfntn} 
Based on this idea, the authors of~\cite{PseudoLeja2012} propose to use~\textit{weakly admissible meshes}, introduced in~\cite{CalviLeastSquares2008}, see also the review~\cite{ReviewWAM2011}. Essentially, these are properly defined adaptive meshes (i.e. that change with~$n$) so that the maximisation condition defining Leja points is replaced by a maximisation over a finite set.

The resulting numerical methods suffer from one significant drawback: they lack modularity, since the construction of a weakly admissible mesh strongly depends on the chosen compact $K$. 

\bigskip 
\textit{\underline{Goal of the present work}.} Our purpose is to propose and analyse several random alternatives to Leja points, and prove that they constitute good points for interpolation. We place emphasis on methods that rely on sampling points uniformly at random inside of $K$.

\paragraph{Random Leja points.}  
 Let $\sigma$ be a Borel measure on $K$.
First, we introduce a new set of points which we call \textit{random Leja points}. Given $Z_0$ any given random variable, they are defined recursively by
\begin{equation}
\label{def_rand_leja}
 Z_n \sim |\pi_n(z)| = \prod_{i=0}^{n-1} |z-Z_i|,
 \end{equation}
where the above notation means that, conditionally on $(Z_0, \ldots, Z_{n-1})$,  the law of $Z_n$ is absolutely continuous with respect to $\sigma$, with density proportional to the function $|\pi_n|$. 

When it comes to actually computing such points, rejection sampling is a natural approach. It has the key property that the normalisation constant $\|\pi_n\|_1 = \textstyle \int_K |\pi_n(z)| d \sigma(z)$ need not be computed. The most natural majorising measure is the uniform measure (with respect to $\sigma$), which we denote~$\mathcal{U}_\sigma(K)$.

In order to apply rejection sampling with the uniform measure, one needs to obtain bounds of the form 
$\|\pi_n\|_K \leq M_n$. Crude estimates such as $M_n = \mathrm{diam}(K)^n$ will become exponentially bad and lead to practically endless rejection loops. On the other hand, one can obtain better (polynomial) bounds by considering the normalisation constant $\|\pi_n\|_1$ by means of so-called~\textit{Nikolskii inequalities}, but the rejection step will require evaluating the constant.

\paragraph{Random Leja points by Metropolis Hastings sampling.}

An alternative approach that can boast the polynomial bound for the normalised density, without computing the normalisation constant $\|\pi_n\|_1$ is the (independent) Metropolis Hastings algorithm, with the uniform measure $\mathcal{U}_\sigma(K)$ as a proposal distribution for ease of implementability. The price to pay is that the algorithm, when stopped at a given iteration $k$, yields a random variable that only approximately follows the law of interest. 

Let us informally describe the (independent) Metropolis Hastings algorithm to sample from a probability density $f$ with proposal distribution given by $\mathcal{U}_\sigma(K)$. Starting from some initial distribution, say $Z_0 \sim \mathcal{U}_{\sigma}(K)$, the algorithm iteratively computes
\begin{align}
\label{informal_mh}
\begin{cases}
X \sim \mathcal{U}_\sigma(K) \\
Z_{k} = 
\begin{cases} 
 X \text{ with probability } \min\bigg(\frac{f(X)}{f(Z_{k-1})}, 1\bigg) \\
 Z_{k-1} \text{ otherwise} 
\end{cases}
\end{cases}
\end{align}
Let us denote $\mathcal{M}_k(f)$ the iterates defined by~\eqref{informal_mh}.


Given $Z_0$ any given random variable, well-chosen numbers $N_n$ for all $n \geq 1$, we recursively define (under appropriate independence assumptions)
\begin{align}
\label{def_MH_Leja}
Z_n := \mathcal{M}_{N_n}\bigg( \frac{|\pi_n|}{\|\pi_n\|_1}\bigg).
\end{align}
The idea is that for a sufficiently large number of iterates $N_n$, the convergence of iterates to the desired distribution should ensure that MH points share any good property random Leja points might have.
We coin these \textit{Metropolis-Hastings random Leja points} and use the abbreviation "MH  points".

Note that the resulting method~\eqref{def_MH_Leja} hence relies on one single main assumption, which is to be able to efficiently draw points uniformly at random (with respect to $\sigma$) in $K$.  This is a very weak requirement that makes computing MH points quite flexible.

\paragraph{Pseudo Leja points by exhaustive uniform sampling.}
An equally flexible approach to the above is to compute pseudo Leja points with a randomised weakly admissible mesh. Such meshes have been introduced and analysed in~\cite{RandomisedWAM2023}.  
Given $Z_0$ any given random variable, well-chosen numbers $N_n$ for all $n \geq 1$ and under appropriate independence assumptions, they are defined recursively (if sampling the randomised mesh according to the uniform measure) by
\begin{align}
\label{def_rand_leja_tris}
Z_n   \in \argmax_{1 \leq k \leq N_n}  |\pi_n(X_k)|, \qquad  X_1, \ldots, X_{N_n} \sim \mathcal{U}_\sigma(K).
\end{align}
Since these are based on a random mesh, we will use the abbreviation "RM points" when speaking about these random points. 
\bigskip

As discussed later on in much more detail, we show that abstract random Leja points~\eqref{def_rand_leja}, as well as the more implementable inspired version thereof~\eqref{def_MH_Leja} (MH points) or the alternative one~\eqref{def_rand_leja_tris} (RM points) with $N_n$ appropriately chosen \textbf{almost surely} provide good points for interpolation, under generic assumptions.

\subsection{State of the art}
\paragraph{Interest in Leja points.} Compared to alternative families of points used in potential theory or polynomial interpolation, such as Chebychev or Fekete points, Leja points have two very appealing properties. First, they are computationally much more tractable. Second, they are hierarchical: when going from $n$ points to $n+1$ points, a single new point has to be computed rather than computing $n+1$ new points from scratch.

\paragraph{Potential theory.} In potential theory, points $(z_n)$ of interests are notably those for which one can recover
 the \textit{logarihmic capacity} $\mathrm{cap}(K)$ of $K$ and, assuming it is unique, the \textit{equilibrium measure} $\mu_K$ of $K$. We refer to~\cite{Ransford1995} for a definition of these two notions.
 For nonpolar compact sets $K$ (i.e., those having positive logarithmic capacity), $\mu_K$ is unique and it is known that Leja points \textit{asymptotically distribute according to the equilibrium measure~$\mu_K$}, namely
\begin{equation}
\label{recover_measure}
 \frac{1}{n} \sum_{j=0}^{n-1} \delta_{z_j} \xrightharpoonup{\; \,\ast\;\, } \mu_K,
\end{equation}
where the convergence above is to be understood in the weak-$\ast$ topology of Radon measures (the topological dual of the set of continuous functions $f : K \to \R$).

\paragraph{Polynomial interpolation.} Let $L_n(f)$ denote the unique polynomial in $\C_{n-1}[X]$ interpolating a given function $f : K \rightarrow \C$, $f \in \mathcal{C}(K)$ at distinct points $z_0, \ldots z_{n-1}$ in $K$, i.e., 
\[L_n(f) = \sum_{i=0}^{n-1} f(z_i)\, \ell_{i,n},  \qquad  \ell_{i,n}(X) = \prod_{j\neq i} \frac{X - z_j}{z_i -z_j}. \]

One essential requirement for points to be good for interpolation is for them to be \textit{extremal}.
\begin{dfntn} 
A sequence $(z_n) \in K^\N$ of distinct points is said to be a set of extremal points for $K$ if it satisfies 
 \[\left(\text{for any $f$ holomorphic in a neighbourhood of $K$}, \quad \|L_n(f) -f \|_K \longrightarrow 0 \right). \]
\end{dfntn}
It is a general fact that distinct points asymptotically distributing according to $\mu_K$ (i.e., satisfying~\eqref{recover_measure}) in a nonpolar compact set are extremal~\cite{Gaier1987lectures}.\footnote{In fact, they are extremal not only for $K$ but even for the polynomial convex hull of $K$, but these two sets coincide for most usual compact sets.}. As a result, Leja points are extremal.

In order to obtain and quantify convergence for less smooth functions, the proper notion is that of the associated \textit{Lebesgue constants} $\Lambda_n$\footnote{This constant is nothing but the operator norm of the continuous linear mapping $f\mapsto L_n(f)$ from $\mathcal{C}(K)$ to itself, endowed with~$\|\cdot\|_K$. Hence, it also characterises how stable the interpolation process is.}, defined by
\[\Lambda_n := \|\lambda_n\|_K, \qquad \lambda_n(z):=  \sum_{i=0}^{n-1} |\ell_{i,n}(z)|, \]
since one has $\textstyle \|L_n(f)-f\|_K \leq (1+ \Lambda_n) \inf_{P \in \C_n[X]} \|P-f\|_K$.

Perhaps surprisingly, few results are known about Lebesgue constants of Leja points. Very recently, Leja points have been proved to have subexponential Lebesgue constants whatever the nonpolar compact set~\cite{Totik2023}, extending the results of~\cite{Totik2010}. Note that points with subexponential Lebesgue constants are extremal in nonpolar compact sets.

Precise estimates are known only for specific compact sets. For instance, one has $\Lambda_n = O(n)$ for the unit circle~\cite{Lebesgue_Leja_Disk2013}, and $\Lambda_n = O(n^{13/4})$ for (finite unions of) closed intervals $K \subset \R$~\cite{Explicit_Lebesgue_Interval2022}, although numerical evidence suggests that this can be improved to $\Lambda_n = O(n)$. 




\paragraph{Pseudo Leja points.}
In order to actually compute Leja points, one crude approach is to fix a fine grid of $K$ once and for all, and compute points on that grid. The resulting points are no longer guaranteed to be extremal; when $n$ nears the grid size, they may even perform quite poorly. Adaptive meshes have been proposed but for very specific compact sets~\cite{FastLeja1998}. 

Pseudo Leja points introduced in~\cite{PseudoLeja2012} are shown to satisfy~\eqref{recover_measure} whenever $K$ is nonpolar, hence they are extremal. In fact, the result in~\cite{Totik2023} shows that Lebesgue constants are subexponential for pseudo Leja points with $\tau_n = \tau$ (that is, pseudo-Leja points of order $0$), whatever the value of $0 < \tau \leq 1$. 
The results of~\cite{Explicit_Lebesgue_Interval2022} for subsets of the real line also accommodate this case in the form $\Lambda_n = O(n^{13/4 + c(\tau)})$ with an explicitly known constant~$c(\tau)$.



\subsection{Main results}

\paragraph{Extremality.} Our main results may informally be stated as follows: for generic (nonpolar) compact sets $K$ and measures~$\sigma$, almost surely
\begin{itemize}

\item[(i)] random Leja points defined by~\eqref{def_rand_leja} are extremal, see Theorem~\ref{main_1},

\item[(ii)] MH points defined by~\eqref{def_MH_Leja} with $N_n$ sufficiently large are extremal, see Theorem~\ref{thm_mh_leja},

\item[(iii)] RM points defined by~\eqref{def_rand_leja_tris} with $N_n$ sufficiently large have subexponential Lebesgue constant, see Theorem~\ref{thm_leja_mesh}. Hence, they are extremal. 
\end{itemize}
In the case of the two last results (ii) and (iii), 'sufficiently large' refers to $N_n$ given in the form $N_n \sim n^\alpha$ with~$\alpha$ dependent on Markov, Nikolskii or covering exponents relative to the compact $K$ (and the measure $\sigma$). For typical compact sets of interests, (upper bounds for) these are available in the literature. For the method to work for any reasonable compact set, one can take any $\alpha>2$ in the case (ii) and any $\alpha>4$ in the case~(iii), but these can usually be improved for many specific compact sets of interest. These should be compared with pseudo-Leja points of~\cite{PseudoLeja2012} which require at most $\alpha = 2$ for generic compact sets. 

For simple subsets of the real line, we may even prove that Lebesgue constants associated to points~\eqref{def_rand_leja_tris} have polynomially growing Lebesgue constant, but we are not able to provide a universal exponent since the corresponding estimate is draw-dependent.

All three results (i), (ii) and (iii) rely on establishing that the points almost surely are pseudo Leja points of some order, and hence asymptotically distribute according the the equilibrium measure $\mu_K$ according to Theorem 1 of~\cite{PseudoLeja2012}. The more precise last result (iii) is based on Theorem~1 of~\cite{Totik2023}, since we show that RM points defined by~\eqref{def_rand_leja_tris} are almost surely pseudo Leja points of order~$0$.

\paragraph{Logarithmic capacity.} 
In fact, a close inspection of how it is proved that pseudo Leja points are extremal~\cite{Bloom1992, PseudoLeja2012, Totik2023} shows that a byproduct of our results are the almost sure convergences
\[ \|\pi_n\|^{1/n} \longrightarrow \,\mathrm{cap}(K), \qquad \text{and} \qquad \bigg(\prod_{0 \leq i < j \leq n-1} |Z_i -Z_j|\bigg)^{\frac{2}{n(n-1)}}\longrightarrow \;  \mathrm{cap}(K).\]
In other words, all the proposed points almost surely can be used to recover the logarithmic capacity of $K$, which is of interest for potential theory. Indeed, logarithmic capacities are notoriously hard to compute, even for simple sets~\cite{NumericalCapacity2008, NumericalCapacity2014}.

The first limit above also explains why, in general, any crude bound like $\|\pi_n\|_K \leq \mathrm{diam}(K)$ will be impractical for implementing random Leja points by rejection sampling. Indeed, we have $\mathrm{cap}(K) \geq \tfrac{1}{4}\mathrm{diam}(K) $ for all connected compact sets $K$ and the inequality is usually strict~\cite{Tsuji1959}.

\paragraph{Separation.} 
When studying these different sets of points, we also establish results pertaining to their~\textit{separation properties}, as separation bounds play an important role in many proofs of estimates for Lebesgue constants, see for instance~\cite{Explicit_Lebesgue_Interval2022, Totik2023}.
We will say that given points $(z_n)$ \textit{satisfy a separation of order $\alpha$} if 
\begin{equation}
\label{min_separation}
 \min_{0 \leq j \leq n-1} |z_n -z_j|  \geq c \, n^{-\alpha},
\end{equation} 
for some constants $c>0$, $\alpha>0$.
We prove that all proposed points almost surely satisfy a separation of (uniform) order, which again is given in terms of the aforementioned exponents.

\paragraph{Numerical results.} 

We run numerical experiments for the interval, the disk, and (potentially elaborate) polygons. For most compact sets, we explain why computing the $n$th MH or RM Leja point requires (at most) of the order of $n^2$ points, except for polygons for which RM points require about $n^4$ points, rendering them intractable. For such sets, the best approach known to us are properly defined pseudo Leja points which require $O(n)$ points (see~\cite{PseudoLeja2012}), but can be cumbersome to compute as they rest on parameterising each edge of the polygon. 

In all cases, numerical experiments suggest that MH and RM points are associated to (almost surely) polynomially growing Lebesgue constants. Whenever MH points and RM points have the same complexity (hence excluding polygons), we consistently find that RM points typically lead to better Lebesgue constants, and can be computed slightly faster. As should be expected, RM points have computation times close to those of pseudo-Leja points based on an underlying admissible mesh of comparable size.

\subsection{Open questions and perspectives}

\paragraph{Generality with respect to the compact set.}
Our results apply to generic compact sets with the natural corresponding measure, as long as Markov (for RM points) or Nikolskii (for MH points) inequalities are of avail. All connected compact sets satisfy a Markov inequality with exponent at most $2$, but there are comparably fewer results in the case of Nikolskii inequalities. For instance, we are not aware of results for sets with cusps. In such a case, MH points are still well-defined but the lack of a Nikolskii inequality prevents us to provide a theoretically-validated rule to choose the number of iterates within the Metropolis-Hastings algorithm.

\paragraph{Sampling method for random Leja points.}
We have found rejection sampling to be computationally infeasible in order to sample random Leja points. This has led us to MH points, which are not random Leja points but an approximation thereof. We do not know if alternative sampling strategies could be competitive with the method developed here.


\paragraph{Curse of dimensionality.}
Our proposed sets of points have a natural generalisation to higher dimensions. In this setting, however, it is not even known whether Leja points are extremal or not. Also, the number of uniform random points $N_n$ one should draw then would be of the order $n^{r d}$ both for MH and RM points, where $r$ depends on the chosen method.

In fact, it is known that weakly admissible meshes (on which RM points are based) necessarily require a number of points that grows exponentially with the dimension~\cite{Kroo2011}. This is why randomised approaches based on weakly admissible meshes RM are bound to suffer from the so-called \textit{curse of dimensionality}. As for MH points, we believe that our estimates based on Nikolskii inequalities are sharp; since Nikolskii exponents typically increase linearly with dimension, the same issue arises. 


\paragraph{Lebesgue constants.}
As already mentioned, the asymptotic behaviour of Lebesgue constants of Leja points (or variants) is poorly understood.  Our attempts at analysing Lebegue constants for the proposed random sets of points were unsuccessful. Finding general estimates for Lebesgue constants of all the points introduced in the present paper is thus a completely open problem.


\subsection*{Outline of the paper}
In Section~\ref{sec_random}, we give the main hypotheses used throughout about the compact $K$ and the measure~$\sigma$. We then proceed to analysing the abstract random Leja points~\eqref{def_rand_leja}.
Then in Section~\ref{sec_MH_Leja}, we analyse the set of MH points~\eqref{def_MH_Leja}, building upon geometric convergence results for independent Metropolis-Hastings samplers in the Wasserstein metric. Then in Section~\ref{sec_random_uniform}, we discuss the case of RM points~\eqref{def_rand_leja_tris}, following ideas of~\cite{RandomisedWAM2023}.
 Finally in Section~\ref{sec_num}, we confirm our theoretical results by numerical experiments for various compact sets and discuss the advantages and caveats associated to random Leja points and variants. 

\section{Random Leja points}
\label{sec_random}
\subsection{Notations and main hypotheses}
\label{subsec_hyp}
Let $K$ be a compact subset $\C$ and $\sigma$ be a finite Borel measure on $K$. We denote \[\|f\|_K := \sup_{z \in K} |f(z)|\] for $f \in \mathcal{C}(K)$, the set of continuous functions $f: K \to \C$.

For $1 \leq p \leq +\infty$, we will be led to consider the $L^p$ spaces associated to the measure~$\sigma$ on $K$. We let $\|\cdot\|_p$ be the corresponding norms.

 From now on, a random variable $Z$ will refer to any measurable mapping $Z : \Omega \to K$ with $(\Omega, \mathcal{A}, \mathbb{P})$ a probability space, and $K$ endowed with the Borel $\sigma$-algebra and measure $\sigma$. 


Generically, if $(Z_n)_{n \in \N}$ is a sequence of random variables, $\mathcal{F}_n$ will stand for the $\sigma$-algebra generated by $Z_0, \ldots, Z_{n-1}$, and $\pi_n$ will stand for the (random) polynomial $\textstyle z \mapsto \prod_{j=0}^{n-1} (z-Z_j)$.

\paragraph{Hypotheses on $K$ and $\sigma$.} We will be using the following hypotheses that $K$ and/or $\sigma$ should satisfy. For all practical purposes, all the hypotheses below are satisfied for most of generic compact sets with their associated 'natural' measures. 

First, we consider an hypothesis coming from potential theory, namely
 \begin{equation} \text{$K$ is nonpolar.} \tag{H1} \label{nonpolar}\end{equation}
All 'reasonable' compact sets are nonpolar: smooth Jordan curves, convex sets and sets with positive area measure are nonpolar.
Under~\eqref{nonpolar}, $K$ admits a unique equilibrium measure $\mu_K$~\cite{Ransford1995}.

 We also need the measure to give weight to the whole of $K$ in the following sense:
  \begin{equation} \text{for all continuous functions  $f :  K \rightarrow \R, \qquad \|f\|_K = \|f\|_\infty$.} \tag{H2} \label{linf}\end{equation}
 For instance, this holds true if $\sigma(O)>0$ for any open set $O$ (in the topology of $K$).
 Assumption~\eqref{linf} is critical in comparing the quality of the newly generated random point, which requires comparing $|\pi_n(Z_n)|$ to $\|\pi_n\|_K$, while the generation process involves $\sigma$ and corresponding estimates naturally involve the $L^\infty$ norm associated to $\sigma$.

 Finally, we will need two types of inequalities, namely a so-called \textit{Markov inequality}
  \begin{equation}  
       \label{nikolskii}
        \forall P \in \C_n[X], \quad \|P\|_\infty \leq c_{\ell}  \, n^{r_{\ell}} \|P\|_1 \tag{H3},\end{equation}
 with $c_\ell>0$, $r_\ell>0$ as well as a so-called \textit{Nikolskii inequality}
  \begin{equation} \label{markov}
   \forall P \in \C_n[X], \quad \|P'\|_K \leq c_m  \, n^{r_m}  \|P\|_K \tag{H4} ,\end{equation}
   with $c_m>0$, $r_m>0$.  
  We call $r_{\ell}$ the \textit{Minkowski exponent} of $(K, \sigma)$, and $r_m$ the \textit{Markov exponent} of $K$, with a slight abuse since these usually refer to the best such possible constants. 
  

For most compact sets $K$ (and the natural corresponding measures $\sigma$), the Nikolskii inequality~\eqref{nikolskii} is known to hold and the (optimal) constant $r_\ell$ is known~\cite{Pritsker1997}.
For $K$ a Jordan domain with smooth boundary and $\sigma$ the area measure, one has $r_\ell = 2$.\footnote{It suffices that the boundary be a quasidisk; this covers smooth boundaries up to Lipschitz regularity, but excludes cusps.} For $K$ a rectifiable Jordan curve with $\sigma$ the arclength, one has $r_\ell = 1$. Line segments have Nikolskii exponent $r_\ell = 2$~\cite{BookPolynomials1994}.

For most compact sets $K$, the Markov inequality~\eqref{markov} is known to hold and the (optimal) constant $r_m$ is known~\cite{ReviewPolynomials1994, BernsteinMarkovReview2021}. For instance, that any compact connected set $K$ satisfies~\eqref{markov} is proved in~\cite{MarkovPommerenke1959} with $r_m = 2$, while for $C^2$-smooth Jordan curves $K$, one has $r_m = 1$. Hence, $r_m = 1$ also for a set $K$ enclosed by a smooth Jordan curve, by the maximum modulus principle.
 
\begin{rmrk}
One readily obtains corresponding results for (finite) unions of sets satisfying the above hypotheses.
\end{rmrk}

We shall consider one additional basic hypothesis to ensure that some of the random points $(Z_n)$ we generate are almost surely distinct:
 \begin{equation} \text{$\sigma$ has no atom.} \tag{H5} \label{atom} \end{equation}
 This assumption is essentially technical and could be relaxed, but at the price of more convoluted statements.

  \subsection{Preparatory results} 
  We gather important results that will play an important role going forward. 
   First, we will find it convenient to use the following basic estimate, which is also used in~\cite{PseudoLeja2012}[Theorem 3].   For completeness, we provide the proof. 

 \begin{lmm}
 \label{taylor}
Assume that~\eqref{markov} holds.
Then for all $P \in \mathbb{C}_n[X]$, one has 
\[\forall u, v \in K, \quad |P(v) -P(u)| \leq \|P\|_K \big( e^{c_m n^{r_m} |v-u|} - 1\big).\]
 \end{lmm} 
 \begin{proof}
A repeated application of Markov's inequality~\eqref{markov} entails $\|P^{(k)}\|_K  \leq c_m^k n^{r_m k} \|P\|_K$.
Using Taylor's formula for polynomials, we deduce that
 \begin{align*} 
|P(v) -P(u)| & \leq \sum_{k=1}^n |P^{(k)}(u)| \frac{|v-u|^{k}}{k!} \leq \sum_{k=1}^n \|P^{(k)}\|_K \frac{|v-u|^{k}}{k!}  \\
& \leq \|P\|_K \sum_{k=1}^n  \frac{\left(c_m n^{r_m} |v-u|\right)^{k}}{k!} 
 =  \|P\|_K\big( e^{c_m n^{r_m} |v-u|} - 1\big)
 \end{align*}
  \end{proof}
  
  Now we recall some of the main results of~\cite{PseudoLeja2012}.
  \begin{thrm}[\cite{PseudoLeja2012}]
  \label{FromCalvi}
  Assume that~\eqref{nonpolar} holds.
If points $z_n$ are pseudo-Leja points for $K$, then
\begin{itemize}
\item  the points $z_n$ asymptotically distribute according to $\mu_K$ and hence are extremal,
\item if furthermore~\eqref{markov} holds, and the points $z_n$ are pseudo-Leja points of order $\beta$, then they satisfy a separation of order $\beta + r_m$.
\end{itemize}
  \end{thrm}
  
 We now provide a lemma that underlies our main results.
\begin{lmm}
 \label{omnipotent_lemma}
 Assume that~\eqref{nonpolar} holds.
Let $(Z_n)$ be a sequence of random variables such that the points $(Z_n)$ are almost surely distinct, and, for some $\beta \geq 0$, \begin{equation}
\label{conv_series}
\sum \mathbb{P} \; (|\pi_n(Z_n)|  < n^{-\beta} \|\pi_n\|_K) \quad \text{converges}.
\end{equation}

Then, almost surely,
\begin{itemize}
\item the points $Z_n$ asymptotically distribute according to $\mu_K$ and hence are extremal,
\item if~\eqref{markov} holds, the points $Z_n$ satisfy a separation of order $\beta + r_m$.
\end{itemize}
 \end{lmm} 
\begin{proof}
By Borel-Cantelli's Lemma, the convergence of the series shows that, for almost every $\omega \in \Omega$ the points $(Z_n(\omega))$ satisfy $|\pi_n(Z_n(\omega))|  \geq n^{-\beta} \|\pi_n\|_K$ for $n$ large enough, say $n \geq n_0(\omega)$, where we slighty abuse notation since $\|\pi_n\|_K$ is also random. Upon restricting to draws $\omega \in \Omega$ such that the points $(Z_n(\omega))$ are distinct, this proves that for almost every $\omega \in \Omega$, the points $(Z_n(\omega))$ are pseudo Leja points of order $\beta$: indeed, one chooses $\tau_n(\omega) = n^{-\beta}$ for $n\geq n_0(\omega)$ and $\tau_n(\omega) = \tau$ with $0 < \tau < 1$ sufficiently small for $n< n_0(\omega)$. That one may choose a positive such $\tau$ is due to the fact that the points are distinct.  

Now the claim of the Lemma follows at once by applying Theorem~\ref{FromCalvi}.
%
%
\end{proof}

 \subsection{Extremality of random Leja points}
 
 Recall that random Leja points are defined as follows.
 \begin{dfntn}
 \label{def_rand_leja_formal}
 We call random Leja points any sequence $(Z_n)$ of random variables such that for all $n \geq 1$, conditionally on $(Z_0, \ldots, Z_{n-1})$, the law of $Z_n$ is absolutely continuous of density (proportional to) $\textstyle z \mapsto |\pi_n(z)| = \prod_{j=0}^{n-1} |z-Z_j|$, with respect to the measure $\sigma$.
  \end{dfntn}
  
Note that, by construction, random Leja points are almost surely distinct. 

\begin{lmm}
\label{lemme_truc}
Assume that~\eqref{nikolskii} holds and let $Z_0, \ldots, Z_{n-1}$ be any random variables. Let $Z_n$ be absolutely continuous of density (proportional to) $\textstyle z \mapsto |\pi_n(z)| = \prod_{j=0}^{n-1} |z-Z_j|$, with respect to the measure $\sigma$.
Then for all $0 < \tau \leq 1$,
\[\mathbb{P}\left(|\pi_n(Z_n)| < \tau  \|\pi_n\|_\infty \, \Big|  \, \mathcal{F}_n \right) \leq  \tau \sigma(K) c_\ell n^{r_\ell}.\]
\end{lmm}
\begin{proof}
We start by writing
\begin{align*} \mathbb{P}\left(|\pi_n(Z_n)| < \tau  \|\pi_n\|_\infty \, \Big|  \, \mathcal{F}_n\right) & = \mathbb{P}\left(|\pi_n(Z_n)|^{-1}> \tau^{-1}  \|\pi_n\|_\infty^{-1} \, \Big|  \,  \mathcal{F}_n\right), \\
&  \leq \tau \|\pi_{n}\|_\infty \; \mathbb{E}\left(|\pi_n(Z_n)|^{-1}  \, \Big|  \, \mathcal{F}_n\right),
\end{align*}
by Markov's (probability) inequality.
Hence, by the very definition of random Leja points, we obtain
\begin{align*} \mathbb{P}\left(|\pi_n(Z_n)| < \tau  \|\pi_n\|_\infty \, \Big|  \, \mathcal{F}_n\right) &
= \tau  \|\pi_{n}\|_\infty \frac{1}{\|\pi_n\|_{1}} \int_{K}|\pi_n(z)|^{-1}  |\pi_n(z)| \, d \sigma(z)  \\
& = \tau \sigma(K) \frac{\|\pi_{n}\|_\infty}{\|\pi_n\|_{1}} 
\end{align*}
Using \eqref{nikolskii}, we find
\begin{align*} \mathbb{P}\left(|\pi_n(Z_n)| < \tau  \|\pi_n\|_\infty \, \Big|\,   \mathcal{F}_n \right)  \leq \tau  \sigma(K) c_\ell n^{r_\ell}.
\end{align*}
\end{proof}

The previous lemma together with Lemma~\ref{omnipotent_lemma} allows us to prove that random Leja points are almost surely pseudo Leja points of order $1+r_\ell + \e$ for any $\e>0$.
\begin{thrm}
\label{main_1}
Assume that~\eqref{nonpolar}, \eqref{linf} and \eqref{nikolskii} hold. 

Then, almost surely, random Leja points asymptotically distribute according to $\mu_K$.
In particular, random Leja points are almost surely extremal.
\begin{proof}
As already mentioned,  the points $(Z_n)$ are almost surely distinct.
Since we are assuming~\eqref{nonpolar}, the expected result will be proved if we establish an estimate of the form~\eqref{conv_series}.
We let $\beta :=  1+ r_\ell+ \e$ for some fixed $\e>0$.

Using $\|\pi_n\|_K = \|\pi_n\|_\infty$ thanks to~\eqref{linf} and taking the expectation in the estimate of Lemma~\ref{lemme_truc} (valid since~\eqref{nikolskii} is assumed to hold), we have for all $n \geq 1$
\[\mathbb{P}(|\pi_n(Z_n)| < n^{-\beta} \|\pi_n\|_K)  = \mathbb{P}(|\pi_n(Z_n)| < n^{-\beta} \|\pi_n\|_\infty )\leq C n^{-\beta} n^{r_\ell} = C n^{-(1+\e)},\] for some constant $C>0$. This shows the expected convergence.
\end{proof}
\end{thrm}

The previous proof  combined with Lemma~\ref{omnipotent_lemma} leads to the following result as to how well random Leja points are separated. 
\begin{crllr}
Assume that~\eqref{nonpolar}, \eqref{linf}, \eqref{nikolskii} and \eqref{markov} hold. 
Then, for any $\e>0$, random Leja points almost surely satisfy a separation of order $1 + r_m + r_\ell + \e$.
\end{crllr}

\section{MH  points}
\label{sec_MH_Leja}
 \subsection{Background}

\paragraph{Wasserstein distance.} For $\mu$ and $\nu$ two Borel measures on $K$, we define the $1$-Wasserstein distance between $\mu$ and $\nu$ by
\begin{align*}
W(\mu, \nu) &= \inf_{\xi \in \Pi(\mu,\nu)} \int_{K\times K}  |u-v| \,  d \xi(u,v),
\end{align*} where $\Pi(\mu, \nu)$ is the set of probability measures over $K\times K$ whose marginals are $\mu$ and $\nu$, respectively.

If we generically denote $\mu_Z$ the measure associated to a random variable $Z$, the $1$-Wasserstein distance rewrites equivalently as
\begin{align*}
W(\mu, \nu) &= \inf_{\substack{\mu_{X} = \mu \\ \mu_{Y} = \nu}} \mathbb{E}\big[|X - Y|\big].
\end{align*} 
It is standard that, in the compact case we are dealing with, the infimum in both formulae above is in fact a minimum.

\paragraph{Independent Metropolis-Hastings algorithm.} 
We here recall basic definitions underlying the (independent) Metropolis-Hastings algorithm~\cite{BookMCMC1999}.
Assume we are given $f$ with $f : K \rightarrow \R$ in $L^1(\sigma)$ with $\|f\|_1 =1$, as well as a function $g :  K \rightarrow \R$ in $L^1(\sigma)$ with $\|g\|_1 =1$, such that $g$ is lower bounded by a positive constant on $K$.
We also define
\[\forall u,v \in K, \quad \alpha(u,v) := \min\Big(1, \frac{f(u)g(v)}{f(v)g(u)}\Big),\]
with the convention $\alpha(u,v) = 0$ if $f(u)= f(v)=0$. 
The goal of the independent Metropolis-Algorithm is to ultimately draw from the distribution $f$ by only drawing from the distribution $g$ (we shall write $X \sim g$ to mean that $X$ has density $g$ with respect to $\sigma$). 

Now consider the independent Metropolis-Algorithm with proposal distribution $\mathcal{U}_\sigma(K)$. We let $(X_k)_{k \in \N}$ be i.i.d. random variables with density $g$, and $(U_k)_{k \in \N^*}$ be i.i.d. random variables $\mathcal{U}([0,1])$, with independence between $(X_k)$ and $(U_k)$. Set $Z_0 = X_0$ and define for all $k \geq 1$
\begin{align}
Z_k = \begin{cases} 
X_k & \text{ if }\,  U_k \leq  \alpha(X_k, Z_{k-1}) \\
Z_{k-1} & \text{ else }
\end{cases}
\end{align}
Note that we consider a minor variant compared to the usual setting, since the initial point is not deterministic but a random variable also drawn according to $g$.

We shall make use of the following estimate, which can be found e.g. in~\cite{IMH2024}.

\begin{thrm}[\cite{IMH2024}]
\label{imh_wasserstein}
Assume that there exists $M>0$ such that for a.e. $z \in K$, $f(z) \leq Mg(z)$. Let  $\mu$ be the measure of density $f$ with respect to $\sigma$.
Then 
\begin{equation}
\label{imh_estimate}
W(\mu_{Z_k}, \mu) \leq   \mathrm{diam}(K)  \Big(1- \frac{1}{M}\Big)^k.
\end{equation}
\end{thrm}

\begin{rmrk}
This result is in fact a rather immediate extension of the results of~\cite{IMH2024}. Indeed, these yield the estimate
\[W(\mu_k(u), \mu) \leq  \Big(1- \frac{1}{M}\Big)^k \; \sup_{u \in K} \int_K |u-v| \, f(v) \,  d \sigma(v) \leq  \mathrm{diam}(K)\Big(1- \frac{1}{M}\Big)^k ,\]
where $\mu_k(u)$ is the measure defined by the MH algorithm but starting from the deterministic $Z_0 = u \in K$. 
In~\cite{IMH2024}, this estimate is given for the Lebesgue measure and involves $(1-\e)^k$ with $\e = \inf \tfrac{g}{f}$. As can be seen from how the result is proved thanks to Corollary 4 of~\cite{Tierney1994}, the Lebesgue measure assumption is superfluous, and the relevant quantity can be taken to be (the inverse of) $M = \sup \tfrac{f}{g}$.
Finally, the estimate~\eqref{imh_estimate} with a random initial point $Z_0$ is easily obtained by a conditioning argument.
\end{rmrk}


 \subsection{Extremality of MH points}

Let us now give the fully detailed definition of Metropolis-Hastings random Leja points (MH points).
They are defined recursively as follows. We start from any random variable $Z_0$, typically $Z_0 \sim \mathcal{U}_\sigma(K)$. Then, assuming that $Z_0, \ldots, Z_{n-1}$ have been computed, we use the Metropolis Hastings with proposal distribution $\mathcal{U}_\sigma(K)$, halted at an appropriate index $k$, in order to approach the density of interest, proportional to $|\pi_n|$. 

With the previous notations, we set $g = \tfrac{1}{\sigma(K)}$ and $f = \tfrac{|\pi_n|}{\|\pi_n\|_1}$, and
\[\forall u,v \in K, \quad \alpha_n(u,v) := \min\Big(1, \frac{|\pi_n(u)|}{|\pi_n(v)|}\Big),\]
with the convention $\alpha_n(u,v) = 0$ if $\pi_n(u)= \pi_n(v)=0$. 

The algorithm relies on a sequence of integers $N_n, \; n \geq 1$. For each $n\geq 1$, we let $(X_n^{(k)})_{k \in \N, n \in \N^*}$ be i.i.d. $\mathcal{U}_\sigma(K)$  and  $(U_n^{(k)})_{k \in \N^*, n \in \N^*}$ be i.i.d. $\mathcal{U}([0,1])$, independent from $Z_0$.

Then, we define the sequence of MH points $Z_n$ as follows, starting from the chosen $Z_0$. Assuming that $Z_0, \ldots, Z_{n-1}$ have been computed, we let $\pi_n$ be the corresponding polynomial. Then MH points are defined by $Z_n := Z_{n}^{(N_n)}$ where the sequence $Z_n^{(k)}$ satisfies 
\begin{align}
\label{metropolis_rand_leja}
\begin{cases}
Z_n^{(0)} = X_{n}^{(0)}  \\
Z_{n}^{(k)} = 
 \begin{cases} 
X_n^{(k)} & \text{ if } U_n^{(k)} \leq  \alpha_n(X_n^{(k)}, Z_n^{(k-1)}) \\
Z_{n}^{(k-1)} & \text{ else }
\end{cases}
\end{cases}
\end{align}

\begin{thrm}
\label{thm_mh_leja}
Assume that~\eqref{nonpolar}, \eqref{linf}, \eqref{nikolskii}, \eqref{markov} and~\eqref{atom} hold. Take 
\[N_n \sim n^{\alpha}, \quad \text{with} \quad \alpha > r_\ell.\] 
Then, almost surely, MH points defined by~\eqref{metropolis_rand_leja} asymptotically distribute according to $\mu_K$.
In particular, MH points are almost surely extremal.
\end{thrm}
\begin{proof}
First, we note that assumption~\eqref{atom} is so that the MH points are almost surely distinct. Indeed, we have $Z_n = X_n^{(k)}$ for some (random) index $k$, and the $X_n^{(k)}$ for $n \geq 1$, $1 \leq k \leq N_n$ are all almost surely distinct under~\eqref{atom}.

In order to apply Lemma~\ref{omnipotent_lemma}, we estimate $\mathbb{P}\left(|\pi_n(Z_n)| < \tau  \|\pi_n\|_\infty \right)$ for a given $0 < \tau \leq 1$. To do so, we look for an estimate for the above quantity with $Z_n$ replaced by $Z_n^{(k)}$, with the aim to utimately set $k = N_n$ since $Z_n = Z_n^{(N_n)}$.

Let $\tilde Z_n$ be any random Leja point associated to $Z_0, \ldots, Z_{n-1}$, i.e., of density (proportional to $|\pi_n|$ conditionally on  $Z_0, \ldots, Z_{n-1}$. We will later on specify how $\tilde Z_n$ is coupled to $Z_n^{(N_n)}$.


First, we decompose
\begin{align*} \mathbb{P}\left(|\pi_n(Z_n^{(k)})| < \tau  \|\pi_n\|_\infty \, \Big|  \, \mathcal{F}_n\right) & = \mathbb{P}\left(|\pi_n(Z_n^{(k)})| < \tau  \|\pi_n\|_\infty,  \; \left||\pi_n(Z_n^{(k)})| - |\pi_n(\tilde Z_n)|\right| \leq \tau \|\pi_n\|_\infty
\, \Big|  \, \mathcal{F}_n\right) \\
& + \mathbb{P}\left(|\pi_n(Z_n^{(k)})| < \tau  \|\pi_n\|_\infty,  \; \left||\pi_n(Z_n^{(k)})| - |\pi_n(\tilde Z_n)|\right| > \tau \|\pi_n\|_\infty \, \Big|  \, \mathcal{F}_n\right)
\end{align*}
The first term satisfies 
\begin{align*} \mathbb{P}\left(|\pi_n(Z_n^{(k)})| < \tau  \|\pi_n\|_\infty,  \; \left||\pi_n(Z_n^{(k)})| - |\pi_n(\tilde Z_n)|\right| \leq \tau \|\pi_n\|_\infty
\, \Big|  \, \mathcal{F}_n\right)
& \leq \mathbb{P}\left(|\pi_n(\tilde Z_n)| < 2 \tau  \|\pi_n\|_\infty \, \Big|  \, \mathcal{F}_n\right)    \\
& \leq 2 \, \tau  \sigma(K) c_\ell n^{r_\ell},
\end{align*}
thanks to Lemma~\ref{lemme_truc}.

For the second term, we use the estimate from Lemma~\ref{taylor} followed by Markov's probability inequality to obtain
\begin{align*}  \mathbb{P}\bigg(|\pi_n(Z_n^{(k)})| < \tau  \|\pi_n\|_\infty,  \; \left||\pi_n(Z_n^{(k)})| - |\pi_n(\tilde Z_n)|\right|&  > \tau \|\pi_n\|_\infty \, \Big|  \, \mathcal{F}_n\bigg) \\
& \leq \mathbb{P}\left(\left||\pi_n(Z_n^{(k)})| - |\pi_n(\tilde Z_n)|\right| > \tau \|\pi_n\|_\infty \, \Big|  \, \mathcal{F}_n\right) \\
& \leq   \mathbb{P}\left[e^{c_m n^{r_m} |Z_n^{(k)} - \tilde Z_n|} - 1 >\tau  \, \Big| \mathcal{F}_n \,  \right] \\
& \leq  \frac{c_m n^{r_m}}{\ln(1+\tau)} \,  \mathbb{E}\left[\big|Z_n^{(k)} - \tilde Z_n\big|\, \Big| \mathcal{F}_n \,  \right] \end{align*}
Now, we set $k = N_n$, and we specifically choose $\tilde Z_n$ to be coupled to $Z_n^{(N_n)}$ in such a way that (conditionally on $Z_0, \ldots, Z_{n-1}$), it realises the infimum within the Wasserstein distance, that is,
\[\mathbb{E}\left[\big|Z_n^{(N_n)} - \tilde Z_n\big|\, \Big| \mathcal{F}_n \,  \right]  = W(\mu_{\tilde Z_n^{(N_n)}}, \mu_{\tilde Z_n}).\]
Now, conditionally on $(Z_0, \ldots, Z_{n-1})$, $\tilde Z_n$ has density $\tfrac{|\pi_n|}{\|\pi_n\|_1}$. The assumptions made about independence ensure that the above Wasserstein distance is controlled as per estimate~\eqref{imh_estimate} of Theorem~\ref{imh_wasserstein} by
\[W(\mu_{\tilde Z_n^{(N_n)}}, \mu_{\tilde Z_n}) \leq \mathrm{diam}(K)   \bigg(1- \frac{1}{M_n}\bigg)^{N_n},\]
where $M_n$ is any constant such that $\tfrac{|\pi_n(z)|}{\|\pi_n\|_1} \leq M_n \tfrac{1}{\sigma(K)}$ for a.e. $z \in K$. By~\eqref{nikolskii} (and~\eqref{linf}), we may choose $M_n = \sigma(K) c_\ell n^{r_\ell}$. 

Hence, all in all and after taking the expectation we have derived the estimate
\[\mathbb{P}\left(|\pi_n(Z_n)| < \tau  \|\pi_n\|_\infty\right)  \leq   \tau  \sigma(K) c_\ell n^{r_\ell} +  \mathrm{diam}(K)  \frac{c_m n^{r_m}}{\ln(1+\tau)}  \left(1- \frac{1}{M_n}\right)^{N_n}\]

To conclude, we need to pick $\tau = \tau_n = n^{-\beta}$ with a properly chosen value of $\beta$, so that the series  $\textstyle \sum \mathbb{P}\left(|\pi_n(Z_n)| < \tau_n  \|\pi_n\|_\infty \, \right)$ converges. 
Since $N_n  \sim n^\alpha$ with $\alpha  >r_\ell$, the geometric factor $(1- \tfrac{1}{M_n})^{N_n}$ converges exponentially quickly to $0$. Hence if we take $\beta = r_\ell + 1+ \e$ for some $\e>0$, both terms within the estimate lead to convergent series, and the proof is finished.
\end{proof}

The estimate obtained in the proof above and Lemma~\ref{omnipotent_lemma} combined yield the  following result about separation of MH points. 
\begin{crllr}
Assume that~\eqref{nonpolar}, \eqref{linf}, \eqref{nikolskii} and \eqref{markov} hold. 
Then, for any $\e>0$, MH points almost surely satisfy a separation of order $1 + r_m + r_\ell + \e$.
\end{crllr}

\section{RM points}
\label{sec_random_uniform}

 \subsection{Covering numbers}
 \label{subsec_cov}
 In order to quantify how well uniformly sampled points will cover $K$, we need to discuss covering numbers. In fact, the correct notion of covering numbers should be adapted to $\sigma$.
 We let $\overline{B}(w, \delta)$ stand for the closed ball of center $w$ and radius $\delta$ for the topology in $K$, i.e., the set $\{z \in K, \; |z-w| \leq \delta\}$. 
 
 \begin{dfntn}
 We say that $(w_1, \ldots, w_p) \in K^p$ is a $\delta$-cover for $(K, \sigma)$ if 
 \[ \sigma\bigg(K \setminus \bigcup_{j=1}^p \overline{B}(w_j, \delta)  \bigg) = 0.\]
 We let $N_K(\delta)$ be the associated~\textit{covering number}, i.e., the minimal number of points needed to obtain a $\delta$-cover for $(K, \sigma)$.
 \end{dfntn}
Note that if $(w_1, \ldots, w_p)$ is a $\delta$-cover in the usual sense, that is, when $K \subset  \cup_{j=1}^p \overline{B}(w_j, \delta)$, then it is also a $\delta$-cover for $(K, \sigma)$. In particular, usual covering numbers provide upper bounds for those defined above.

Our assumption regarding covering numbers for $(K,\sigma)$ will be of the form
  \begin{equation}  N_K(\delta) \leq c_c \, \delta^{-r_c} \tag{H6} \label{covering}\end{equation}
for some $c_c > 0$ and $r_c>0$. We call $r_c$ the \textit{covering exponent} of $(K, \sigma)$.
For instance, covering numbers for smooth Jordan curves satisfy $r_c = 1$, and we may take $r_c = 1$ for $1$-dimensional convex sets (line segments), and $r_c = 2$ for $2$-dimensional convex sets.

Now we give two consecutive lemmata; similar arguments can be found in~\cite{RandomisedWAM2023}[Theorem 4.1]. Compared to the latter theorem, our results do not require $K$ to be convex. This is because we are dealing with polynomials for which exact Taylor formulae exist, making estimates such as the one of Lemma~\ref{taylor} possible.

\begin{lmm}
Assume that~\eqref{linf} and~\eqref{markov} hold.
Let $\delta>0$ and $w_1, \ldots, w_{p}$ be $p$ points in $K$ associated to a $\delta$-cover of $K$.

Let $(u_1, \ldots, u_{N_n})$ be $N_n$ points in $K$ such that for any $j \in \{1, \ldots, p\}$, there exists $k \in \{1, \ldots, N_n\}$ such that $|u_k - w_j| \leq \delta$.
Then for any $P \in \C_n[X]$ and $\delta$ small enough,
\[ \|P\|_K \leq \frac{1}{2-e^{2 \delta c_m n^{r_m}}} \max_{1\leq k \leq {N_n}} |P(u_k)|.\]
\end{lmm}
\begin{proof}
By Lemma~\ref{taylor}, for $z \in \cup_{j=1}^{p} \overline{B}(w_j, \delta)$ fixed and any $u_k$
\[|P(z) - P(u_k)| \leq \|P\|_K (e^{c_m n^{r_m} |z-u_k|}-1).\]
Now choose $j$ such that $|z-w_j| \leq \delta$, and then $u_k$ such that $|u_k - w_j| \leq \delta$, so that for that choice of $u_k$ the previous inequality yields
\[|P(z) - P(u_k)| \leq \|P\|_K (e^{2 \delta c_m n^{r_m} }-1),\]
hence 
\[|P(z)| \leq  |P(u_k)|  +   (e^{2 \delta c_m n^{r_m}  }-1)\|P\|_K \leq   \max_{1\leq k \leq {N_n}} |P(u_k)| + (e^{2 \delta c_m n^{r_m}  }-1) \|P\|_K.  \]
This property holds for all $z \in \cup_{j=1}^{p} \overline{B}(w_j, \delta)$, hence by the covering property it holds for a.e. $z \in K$.
Thus we obtain by taking the (essential with respect to $\sigma$) supremum
\[\|P\|_\infty \leq  \max_{1\leq k \leq {N_n}} |P(u_k)| + (e^{2 \delta c_m n^{r_m}  }-1) \|P\|_K.\]
The announced equality follows by using $\|P\|_K = \|P\|_\infty$ thanks to~\eqref{linf}, and by rearranging terms assuming that $\delta$ is small enough (so that $2 - e^{2\delta c_m n^{r_m}} >0$).

\end{proof}

\begin{lmm}
\label{covering_random}
Assume that~\eqref{linf} and~\eqref{markov} hold.
Let $\delta>0$ and $w_1, \ldots, w_{p}$ be $p$ points in $K$ associated to a $\delta$-cover of $K$.
Assume that $X_1, \ldots, X_{N_n}$ are $N_n$ i.i.d. points following $\mathcal{U}_\sigma(K)$.

Then with probability at least $1- p \, e^{- \frac{N_n}{p}}$, there holds for all $P \in \C_n[X]$ and $\delta$ small enough
\[ \|P\|_K \leq \frac{1}{2-e^{2 \delta c_m n^{r_m}}} \max_{1\leq k \leq {N_n}} |P(X_k)|.\]
\end{lmm}
\begin{proof}
By the previous lemma, the result will be proved if we establish that with probability at least $\textstyle 1- p \, e^{-\frac{N_n}{p}}$, for any $j \in \{1, \ldots, p\}$, there exists $k \in \{1, \ldots N_n\}$ such that $|X_k - w_j| \leq \delta$. 

First, we note that the covering property entails 
\[\sigma(K) = \sigma\bigg(\bigcup_{j=1}^{p} \overline{B}(w_j, \delta)\bigg) \leq  \, \sum_{j=1}^{p}  \sigma(\overline{B}(w_j, \delta)),\]

For a fixed $j \in \{1, \ldots, p\}$, we have by the i.i.d. uniform hypothesis
\begin{align*}\mathbb{P}\left(\exists k \in \{1, \ldots, N_n\}, \, X_k \in \overline{B}(w_j, \delta) \right) & = 1- (\mathbb{P}(X_1 \notin \overline{B}(w_j, \delta)))^{N_n}  = 1 - \Big(1-\frac{\sigma(\overline{B}(w_j, \delta)}{\sigma(K)}\Big)^{N_n} \\
& \geq   1 -e^{- \frac{\sigma(\overline{B}(w_j, \delta))}{\sigma(K)} N_n}
\end{align*}
To make notations less cumbersome, let $\sigma_j:= \tfrac{\sigma(\overline{B}(w_j, \delta))}{\sigma(K)}$, so that taking a union bound we have found
\begin{align*}\mathbb{P}\left(\forall j \in \{1, \ldots, p\}, \; \exists k \in \{1, \ldots, N_n\}, \, X_k \in \overline{B}(w_j, \delta) \right) \geq  1 - \sum_{j=1}^{p}  e^{- \sigma_j N_n}.
\end{align*}
Jensen's inequality applied to the concave function $\sigma \mapsto e^{- N_n \sigma}$ leads to 
\begin{align*}\mathbb{P}\left(\forall j \in \{1, \ldots, p\}, \; \exists k \in \{1, \ldots, N_n\}, \, X_k \in \overline{B}(w_j, \delta) \right) \geq  1 - p \, e^{- (\sum_{j=1}^{p} \sigma_j)\frac{N_n}{ p}} \geq 1 - p e^{- \frac{N_n}{p}},
\end{align*}
where the last inequality follows from $\textstyle \sum_{j=1}^{p} \sigma_j \geq 1$.

\end{proof}

\subsection{Pseudo Leja points with randomised mesh}
The algorithm underlying RM points relies on a sequence of integers $N_n, \; n \geq 1$. We let $(X_n^{(k)})_{k \in \N, n \in \N^*}$ be i.i.d. $\mathcal{U}_\sigma(K)$, with the assumption that these are independent from $Z_0$.

Inspired by the randomised weakly admissible meshes introduced in~\cite{RandomisedWAM2023}, the construction we are interested in is given by
\begin{align}
\label{rand_leja_mesh} 
Z_n   \in \argmax_{1 \leq k \leq N_n}  |\pi_n(X_n^{(k)})|.
\end{align}
Note that, assuming~\eqref{atom}, there is (almost surely) no ambiguity in the choice of~$Z_n$.

Thanks to the recent result in~\cite{Totik2023}, we can actually prove a bit more about these points: they almost surely have subexponential Lebesgue constant.
\begin{thrm}
\label{thm_leja_mesh}
Assume that~\eqref{nonpolar}, \eqref{linf}, \eqref{markov}, \eqref{atom},~\eqref{covering} hold. Take
\[N_n \sim n^{\alpha}, \quad \text{with} \quad \alpha > r_m  \, r_c.\] 
 Then almost surely, RM points defined by~\eqref{rand_leja_mesh} have subexponential Lebesgue constant.
In particular, they are almost surely extremal.
\end{thrm}
\begin{proof}
As in the case of MH points, assumption~\eqref{atom} ensures that RM points are almost surely distinct.

By definition, we have $|\pi_n(Z_n)| =\textstyle  \max_{1 \leq k \leq N_n}  |\pi_n(X_n^{(k)})|$.
Let us show that $\textstyle\sum \mathbb{P}(|\pi_n(Z_n)| < \tfrac{1}{2} \|\pi_n\|_K)$ converges. 

Choose $\delta = n^{-\gamma}$ with $r_m < \gamma < \tfrac{\alpha}{r_c}$. For $n$ large enough, one has $\tfrac{1}{2- e^{2 \delta c_m n^{r_m}}} \leq  2$ since $r_m < \gamma$. 
For such $n$'s, we compute 
\[\mathbb{P}\big(|\pi_n(Z_n)|  <   \tfrac{1}{2} \|\pi_n\|_K \, \big| \, \mathcal{F}_n\big) = \mathbb{P}\Big(\max_{1\leq k \leq {N_n}} |\pi_n(X_n^{(k)})| < \tfrac{1}{2} \|\pi_n\|_K\Big| \, \mathcal{F}_n\Big) \leq p(\delta) e^{- \frac{N_n}{p(\delta)}}, \]
for any $p(\delta)$ such that there exist $p(\delta)$ points realising a $\delta$-cover of $(K,\sigma)$.
The inequality follows from applying the previous lemma, which is made possible by the independence between the two random vectors $(Z_0, \ldots, Z_{n-1})$ and $(X_n^{(1)}, \ldots X_{n}^{(N_n)})$. 

By the hypothesis~\eqref{covering} about covering numbers for $(K, \sigma)$, we may choose $p(\delta) = c_c \, \delta^{-r_c} = c_c \, n^{\gamma \, r_c}$. It follows that upon taking the expectation 
\[\mathbb{P}\big(|\pi_n(Z_n)| <   \tfrac{1}{2} \|\pi_n\|_K \big)  \leq p(\delta) e^{- \frac{N_n}{p(\delta)}} \leq c_c \,  n^{\gamma r_c} e^{-  c_c^{-1} N_n n^{- \gamma r_c}}, \]
an estimate showing that the series does converge since $\gamma < \tfrac{\alpha}{r_c}$.

We have proved that points defined by~\eqref{rand_leja_mesh} are so that, for almost every $\omega \in \Omega$, they satisfy the property defining pseudo Leja points, for $n$ large enough (depending on $\omega$), with $\tau_n = \tau = \tfrac{1}{2}$. Since the points $Z_n$ are almost surely distinct, they are pseudo Leja points for almost every $\omega \in \Omega$, for some $0< \tau(\omega) \leq \tfrac{1}{2}$, upon diminishing the value of $\tau$ to account for small values of $n$. The results of~\cite{Totik2023} show that these points have subexponential Lebesgue constant.
\end{proof}

The previous proof  combined with Lemma~\ref{omnipotent_lemma} leads to the following result as to how well points defined by~\eqref{rand_leja_mesh} are separated. 
\begin{crllr}
Assume that~\eqref{nonpolar}, \eqref{linf}, \eqref{markov}, \eqref{atom},~\eqref{covering} hold. Then RM points defined by~\eqref{rand_leja_mesh} almost surely satisfy a separation of order $1 + r_m$.
\end{crllr}

In the case of subsets of the real line, one can make the above result more precise
 thanks to~\cite{Explicit_Lebesgue_Interval2022}.
\begin{prpstn}
\label{prop_leja_mesh}
Assume that $K \subset \R$ is a finite union of closed intervals, and that~\eqref{linf},~\eqref{atom} hold. Take $N_n \sim n^{\alpha}$ with $\alpha > 2$. 

 Then almost surely, RM points defined by~\eqref{rand_leja_mesh} have polynomially growing Lebesgue constants.
\end{prpstn}

\begin{proof}
First, note that such sets are nonpolar. Also, one can fix $r_m = 2$, $r_c = 1$, hence $\alpha > r_m r_c$ becomes $\alpha>2$.
As seen with Theorem~\ref{thm_leja_mesh}, this choice ensures that the points of interest almost surely are pseudo Leja points of order $0$. 

Then by the results of~\cite{Explicit_Lebesgue_Interval2022}, it follows that almost surely in $\omega \in \Omega$, $\Lambda_n(\omega) = O(n^{13/4 + c(\omega)})$ for some unknown constant~$c(\omega)$.
\end{proof}

\section{Numerical experiments}

This section compiles numerical simulations aiming at
\begin{itemize}
\item confirming our theoretical results about extremality of various random points, 
\item investigating properties that remain ouf of reach from the theoretical point of view (mostly, Lebesgue constants),
\item comparing various random families points between themselves, as well as, when it is relevant, with deterministic (pseudo)-Leja points, in terms of accuracy and computational cost.
\end{itemize}
\label{sec_num}
In order to do so, we consider three (types of) compact sets:  the interval $K = [-1,1] \times \{0\}$, the unit disk, and finally polygons, with their natural corresponding arclength or area measures $\sigma$. 

As we shall see, numerical evidence suggests that all our random points lead to almost surely polynomially growing Lebesgue constants.

\subsection{Generalities}
All random families points start with some random variable $Z_0$; we will always take $Z_0 \sim \U_\sigma(K)$. 

\paragraph{Pseudocode for MH points.}
For the sake of reproducibility, we provide below a pseudocode for how to generate the $n$th MH point assuming $z_0, \ldots, z_{n-1}$ have already been computed. In order to evaluate $\pi_n$ as little as possible, it is worth storing the value of $\pi_n(Z)$ where $Z$ stands for the current point along the Metropolis-Hastings iterates.    

\begin{algorithm}[h!]
\caption{Given $\pi_n(z) = \textstyle \prod_{j=0}^{n-1}(z-z_j)$, generation of the $n$th MH point; parameter $N_n$}
\begin{algorithmic}
\State Draw $(X_k)_{0 \leq k \leq N_n} \sim \U_\sigma(K)$ and $(U_k)_{1 \leq k \leq N_n} \sim \U([0,1])$
\State $Z \gets X_0$
\State $\Pi \gets \pi_n(X_0)$
\For{$1 \leq k \leq N_n$} 
    \If{$U_k \leq \tfrac{|\pi_n(X_k)|}{|\Pi|}$}
    	\State $Z  \gets X_k$
        \State $\Pi  \gets \pi(X_k)$
    \EndIf
\EndFor \\
\Return $Z$
\end{algorithmic}
\end{algorithm}
The implementation of RM points is straightforward. Note that, all computed (RM, MH, pseudo-Leja, etc) points require to compare the values taken by $|\pi_n|$ at different points. Doing so is unstable since the product may rapidly take very small values as $n$ increases. Hence in all examples below we in fact rather compare logarithms by means of the sum $\log(|\pi_n|(z)) =\textstyle \sum_{j=0}^{n-1} |z-z_j|$. This is good practice as is well known in the area.

\paragraph{Number of chosen points.} Recall that both MH points or RM points rely on a number of points $N_n$. Our results make it precise how to choose $N_n \sim n^\alpha$ with an appropriately chosen~$\alpha$, in the form $\alpha > r_\ell$ for MH points and $\alpha > r_m r_c$ for RM points. We will throughout take $N_n =\lfloor{n^{\alpha}} \rfloor$ with $\alpha = r_\ell + \e$ (and $\alpha = r_m r_c +\e$, respectively), with the value of $\e$ set to $\e = 0.01$. It is worth stressing that the parameter $\e$ plays an important role. Indeed, by taking larger values of $\e>0$, we would typically obtain "better" points, but at the expense of increased computational cost. 
From now on, any reference to MH or RM points hence refers to the points with our specific choice of numbers $N_n$.

When it comes to pseudo-Leja points, various constructions are introduced in~\cite{PseudoLeja2012}. We will mostly stick to the arguably most general one, given by Proposition 2 there. It essentially applies to nonpolar compact sets with $C^1$ boundary, and rests on a Markov inequality~\eqref{markov}. Assuming a parameterisation of the boundary is known, each point is chosen among a family of $N_n$ points with $N_n \sim n^{r_m}$, and yields pseudo-Leja points of order $0$. For less smooth compact sets, such as polygons, alternatives are proposed in~\cite{PseudoLeja2012} with $N_n \sim n$, yielding pseudo-Leja points that are close to being of order $0$, since with the notations of Definition~\ref{pseudoLeja}, they are pseudo-Leja points with $\tau_n \sim \tfrac{1}{\ln(n)}$. We could also define random points to be sampled directly on the boundary to reach a comparable number of points; we do not pursue this here.

\paragraph{Comparison of methods.}
We provide below a table that summarises the pros and cons of different methods.  
\begin{table}[h!]
\centering
\begin{tabular}{ | c ||c|c|c|}
 \hline \hline
  & MH points  &RM points& \makecell{pseudo-Leja points \\(\cite{PseudoLeja2012}, Proposition 2)}\\
   \hline   \hline
modularity   &  \ding{51} &   \ding{51}  &    \ding{55} \\
   \hline
reproducibility &   \ding{55} & \ding{55} & \ding{51} \\
   \hline
\makecell{order as pseudo-Leja points \\ (accuracy)} &  $\sim 1+r_\ell$  &   0    &    0 \\
   \hline
\makecell{number of underlying points \\(complexity)} &  $r_\ell$  &   $r_m r_c$    &    $r_m$ \\

 \hline
\end{tabular}
\caption{Comparison between different methods.}
\label{comparison}
\end{table}

By modularity, we mean that the random versions we propose are more easily implemented with very little knowledge about the compact $K$ as compared to pseudo-Leja points, since the only underlying assumption is to be able to draw points uniformly at random.

One price to pay with random approaches compared to deterministic ones (such as pseudo-Leja points) is the lack of reproducibility. As will be seen with the upcoming examples, different samples may lead to significantly different Lebesgue constants. Hence a good practice with random approaches would be to sample few examples and choose the best in terms of Lebesgue constant, but finely estimating such constants can be computationally expensive. 

In terms of accuracy, which can informally be measured by the order of points that are known to be pseudo-Leja points, deterministic pseudo-Leja points of~\cite{PseudoLeja2012}[Proposition 2] and RM points are of order $0$, while MH points are order $1+r_\ell+\e$ for all $\e>0$. Theoretically, it is not known how exactly the orders of pseudo-Leja points actually impacts Lebesgue constants, but our numerical simulations seem to show increased Lebesgue constants, as intuition would suggest.  

Each method has a corresponding number of underlying points $N_n$. For many compact sets, we have $r_\ell = r_m r_c$ so that the two random sets of points behaves equivalently in terms of computational burden. For sets with positive area measure and boundary not smoother than Lipschitz, RM  points become intractable, while the two others remain tractable. 

Finally, let us note that only RM points lend themselves to straightforward parallelisation.

\subsection{Interval}
We start with the case of the interval $K = [-1,1] \times \{0\}$ with $\sigma$ the arclength, for which all relevant hypotheses~\eqref{nonpolar}, \eqref{linf}, \eqref{nikolskii}, \eqref{markov}, \eqref{atom} and~\eqref{covering} hold, with $r_\ell = 2$, $r_m=2$ and $r_c = 1$. Hence, theory requires $N_n = n^{2+\e}$ for both MH points and RM points.

\paragraph{Single sample of MH points.} First, we draw one single sample of $500$ MH points, with $N_n = \lfloor{n^{r_\ell+\e}} \rfloor= \lfloor{n^{2+\e}} \rfloor$ underlying points, and we illustrate the result of Theorem~\ref{thm_mh_leja}.
We show on the left panel of Figure~\ref{Single_Draw_Interval} how the measure $\mu_K$ is recovered, which here is known explicitly to have density $x\mapsto \tfrac{1}{\pi} \tfrac{1}{\sqrt{1-x^2}}$ with respect to the arclength.

Then we look at the interpolation process associated to that specific set of points, when applied to the fonction $f :z \mapsto \tfrac{1}{z^2+0.1^2}$. We numerically evaluate the error $e_n(f) := \|L_n(f) -f\|_K$ by computing the maximum on a grid of $10^4$ points. Figure~\ref{Single_Draw_Interval} displays the evolution of $\log(e_n(f))$ as a function of $n$, which shows the expected asymptotic behaviour. That is, we obtain geometric convergence of the error towards $0$ (up until the plateau seen at $n\approx 300$, due to machine precision being reached). In fact, the obtained slope for $n \in \{0, \ldots, 300\}$ by a linear fit corresponds to a geometric convergence of about $\approx 0.914$, close to the theoretically expected $\tfrac{10}{\sqrt{101}+1} \approx 0.905$, see~\cite{PseudoLeja2012} for more details. 

\bigskip 
\begin{figure}[h!]
     \centering
       \hspace*{\fill}%
     \begin{subfigure}[b]{0.35\textwidth}
         \centering
         \includegraphics[width=\textwidth]{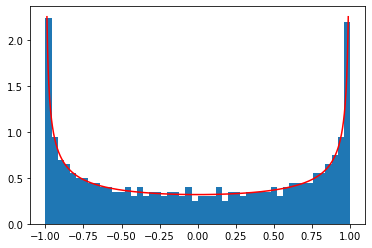}
     \end{subfigure}
     \hfill
     \begin{subfigure}[b]{0.35\textwidth}
         \centering
         \includegraphics[width=\textwidth]{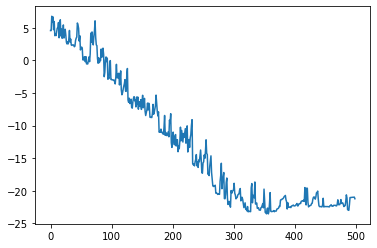}
     \end{subfigure}
            \hspace*{\fill}%
     \hfill
        \caption{\textit{Single draw of $500$ MH points}. On the left, normalised histogram of the 500 points (with a bin size equal to $50$), compared to the density $x\mapsto \tfrac{1}{\pi} \tfrac{1}{\sqrt{1-x^2}}$ the equilibrium measure $\mu_K$ has with respect to the arclength. On the right, evolution of the log-error $\log(\|L_n(f)-f\|_K)$ for $f  : z \mapsto \tfrac{1}{z^2+0.1^2}$, showing geometric convergence of the error until machine precision is reached.} 
    \label{Single_Draw_Interval}
\end{figure}

\paragraph{Computational times.} 
We report the time it takes to sample from either MH points or RM points. We find that the latter points are obtained 4 to 5 times faster. When the number $n$ nears about $1000$, computation times become of the order of the hour, or few hours for MH points. For comparison, it is worth noting that if one were to use weakly admissible meshes with the same number of points, that is, of the order $n^2$, then one would be led to computation times close to those found for RM points.  

These results can be checked to be very robust with respect to the chosen sample. 

\begin{table}[h!]
\centering
\begin{tabular}{ | c ||c|c|c|c|c|c|c|c|}
 \hline
  Number of points $n$ & 100  & 200 & 300 & 400 & 500 & 600 & 700 & 800 \\
 \hline  \hline
 MH points   & 0'14  & 2'01 & 6'50 & 16'39 & 35'42 & 64'51 & 106'19 & 168'21 \\
    \hline
RM points & 0'02  & 0'21 & 1'24 &  5'19 & 9'25 & 15'43 & 25'52 & 40'31    \\ 
    \hline
\end{tabular}
\caption{Time required (in minutes) to compute one sample of $n$ MH or RM points within the interval, with $N_n =  \lfloor{n^{2+\e}} \rfloor$, $\e = 0.01$.}  
\label{times}
\end{table}

\paragraph{Average Lebesgue constants.} 
We then investigate Lebesgue constants for the two proposed random sets of points. In order to do so, we compute $100$ samples of $200$ points, and numerically compute $\Lambda_n$ for $n \in \{0, \ldots, 200\}$. Then, we compute the corresponding statistical averages and standard deviations. 

We emphasise that Lebesgue constants are notoriously hard to properly evaluate; the results presented are obtained by computing the maximum of the corresponding Lebesgue functions $\lambda_n$ on a sufficiently fine grid of the unit disk which, rigorously speaking,  only provides a lower bound for $\Lambda_n$. We used a grid of about $5. 10^4$ points, as we empirically checked that for up to $200$ interpolation points, a finer grid does not lead to significant improvement in approximating~$\Lambda_n$. 
We report the values in Table~\ref{lebesgue_interval}, in the form of estimates for $\mathbb{E}[\Lambda_n]$ and $\sqrt{\mathrm{Var}(\Lambda_n)}$. More precisely, we compute the slope obtained by linear fit of $\log(\mathbb{E}[\Lambda_n])$ (respectively $\log(\sqrt{\mathrm{Var}(\Lambda_n)}$)) against $\log(n)$, for $n \in \{10, \ldots, 200\}$. 
 
In order to explore the difference between MH points and the theoretical random Leja points they originate form, we also approximately compute Lebesgue constants of random Leja points. One sample of random Leja points is obtained as follows: we use rejection sampling by computing an approximate upper bound for $\|\pi_n\|_K$ as follows: we calculate $2 \times \max |\pi_n(w_i)|$ where the $w_i$'s are fixed to be on a grid of $10^4$ points. This upper bound is thus not guaranteed to be an actual upper bound, and will typically fail to be one if $n$ is taken to be too large.

Numerically, we find that all random points almost surely have polynomially growing Lebesgue constants. On average, RM points have the best (average) Lebesgue constants, which grow around~$n^{0.54}$, while MH points have a comparatively worst average Lebesgue constant in $n^{1.56}$. This is comparable to the exponent found for approximate random Leja points, which suggests that the Metropolis-Hastings procedure (with the chosen value of $N_n$) involved in sampling from the distribution $\pi_n$ does not worsen Lebesgue constants.

For the sake of comparison, let us note that (approximate) usual Leja points obtained by maximisation of $|\pi_n|$ on a grid of $10^5$ points, starting from $Z_0$ uniformly chosen at random behave like  $n^{0.71}$ (on average for $100$ draws of $Z_0$), which is comparable to the order obtained for RM points.

We also find that Lebesgue constants consistently have standard deviations that are of the same order as their averages.

\begin{table}[h!]
\centering
\begin{tabular}{ | c ||c|c|c|}
 \hline
  & MH points  &  \makecell{(approximate) \\random Leja points} &RM points\\
 \hline  \hline
 $\mathbb{E}[\Lambda_n]$   &   $1.56$    &    $1.68$ &   $0.54$  \\
    \hline
$\sqrt{\mathrm{Var}(\Lambda_n)}$ &  $1.6$  &    $1.86$ &   $0.50$     \\ 
    \hline
\end{tabular}
\caption{Estimates for polynomial growth of $\mathbb{E}[\Lambda_n]$ and  $\sqrt{\mathrm{Var}(\Lambda_n)}$ in the case of the interval, as obtained by linear fit in $\log$-$\log$ scale, for $n \in \{10, \ldots, 200\}$. Statistical averages are computed over $100$ draws, Lebesgue constants are estimated on a grid of $5. 10^4$ points.  Underlying number of points for MH and RM points are $N_n = \lfloor{n^{2+\e}} \rfloor$ with $\e = 0.01$.} 
\label{lebesgue_interval}
\end{table}

\subsection{Unit disk}
In the case of the unit disk, usual deterministic Leja points are actually well-understood. In fact, they have a closed-form expression~\cite{Calvi_UnitDisk_2011}, and it is known that their Lebesgue constant satisfies $\Lambda_n \leq 2 n$~\cite{Lebesgue_Leja_Disk2013}.

The considered measure $\sigma$ is the area measure; then all hypotheses~\eqref{nonpolar}, \eqref{linf}, \eqref{nikolskii}, \eqref{markov}, \eqref{atom} and~\eqref{covering} hold, with $r_\ell = 2$, $r_m= 1$ and $r_c = 2$. Hence both in the case of MH and RM points, $N_n = \lfloor{n^{2+\e}} \rfloor $ points are required to be drawn. 

\paragraph{Single sample.} First, we show one sample of a set of $200$ points drawn according to the two techniques, see Figure~\ref{Single_Disk}. In the case of MH points, we find more points towards the interior of the disk, while there are fewer for RM points. This is to be expected since the first points are almost surely pseudo-Leja points of order $\sim 1+r_\ell = 3$, while the second points are almost surely pseudo Leja points of order $0$.

\begin{figure}[h]
     \centering
       \hspace*{\fill}%
     \begin{subfigure}[b]{0.34\textwidth}
         \centering
         \includegraphics[width=\textwidth]{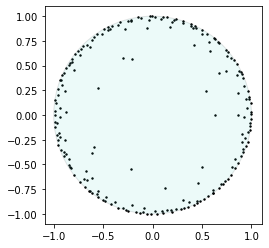}
     \end{subfigure}
     \hfill
     \begin{subfigure}[b]{0.34\textwidth}
         \centering
         \includegraphics[width=\textwidth]{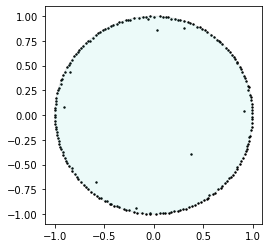}
     \end{subfigure}
       \hspace*{\fill}%
     \hfill
        \caption{Example of $n=200$ MH points (left figure) and RM points (right figure), with $N_n = \lfloor{n^{2+\e}} \rfloor$ in both cases, for $\e = 0.01$.} 
    \label{Single_Disk}
\end{figure}

\paragraph{Average Lebesgue constants.} 
We then investigate average Lebesgue constants, evaluating the corresponding Lebesgue functions on a grid of around $5. 10^4$ points, following exactly the procedure described in the case of the interval.
We uncover that Lebesgue constants are quite significantly worse for MH points than they are for RM points.

As in the case of the interval, we find that average and standard deviation of Lebesgue constants are of the same order.

\bigskip 
\begin{table}[h!]
\centering
\begin{tabular}{ | c ||c|c|}
 \hline
  & MH points  &RM points\\
 \hline \hline
 $\mathbb{E}[\Lambda_n]$   & $2.92$   &   $0.50$  \\ 
    \hline
$\sqrt{\mathrm{Var}(\Lambda_n)}$&   $2.99$ & $0.51$ \\
    \hline
\end{tabular}
\caption{Estimates for polynomial growth of $\mathbb{E}[\Lambda_n]$ and  $\sqrt{\mathrm{Var}(\Lambda_n)}$ in the case of the disk, as obtained by linear fit in $\log$-$\log$ scale, for $n \in \{10, \ldots, 200\}$. Statistical averages are computed over $100$ draws, Lebesgue constants are estimated on a grid of around $5.10^4$ points.  Underlying number of points for MH and RM points are $N_n =\lfloor{n^{2+\e}} \rfloor$ with $\e = 0.01$.} 
\label{lebesgue_disk}
\end{table}

\subsection{Polygons}
We end this subsection by the case of polygons, whose boundary is only Lispchitz. Taking $\sigma$ to be the area measure, all hypotheses~\eqref{nonpolar}, \eqref{linf}, \eqref{nikolskii}, \eqref{markov}, \eqref{atom} and~\eqref{covering} again hold, with $r_\ell = 2$, $r_m=2$ and $r_c = 2$. 

In this case, RM points become computationally intractable when $n$ approaches about $100$, since $r_m r_c =4$. On the other hand, MH points are significantly more tractable; we give an example of one sample of $200$ MH points for two polygons in Figure~\ref{fig_polygons}. 

Finally, the method of Proposition 2 in~\cite{PseudoLeja2012} does not apply here because the boundary is not smooth enough. 
In this case, the best method in terms of number of points $N_n$ is another one proposed in~\cite{PseudoLeja2012}, by using Chebychev points on each edge, leading to pseudo-Leja points with slow discrepancy $\tau_n\sim \tfrac{1}{\log(n)}$. The resulting method requires $N_n = O(n)$ points. 
This improved complexity is made at the loss of modularity: for a complex polygon such as the one used on the right panel of Figure~\ref{fig_polygons}, one would need to parameterise each of the $15$ edges, whereas MH  points merely require to know how to decide if a point is inside the polygon or not, since one can then easily use a rejection method to sample uniformly within the polygon.


Interestingly, points concentrate all around the boundary and more especially at the corners, but only on the smoother part. This would not be the case for the aforementioned deterministic pseudo Leja points, since these would distribute points close to all corners.

\begin{figure}[H]
     \centering
       \hspace*{\fill}%
     \begin{subfigure}[b]{0.34\textwidth}
         \centering
         \includegraphics[width=\textwidth]{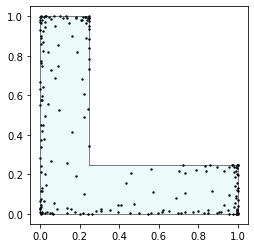}
     \end{subfigure}
     \hfill
     \begin{subfigure}[b]{0.34\textwidth}
         \centering
         \includegraphics[width=\textwidth]{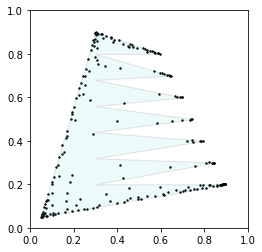}
     \end{subfigure}
       \hspace*{\fill}%
     \hfill
        \caption{Example of $n=200$ MH points for two polygons, with $N_n = \lfloor{n^{2+\e}} \rfloor$, $\e = 0.01$.} 
    \label{fig_polygons} 
\end{figure}

\bibliographystyle{unsrt}
\bibliography{Biblio_Leja.bib}

\begin{thebibliography}{10}

\bibitem{Leja1957}
Franciszek Leja.
\newblock Sur certaines suites li{\'e}es aux ensembles plans et leur
  application {\`a} la repr{\'e}sentation conforme.
\newblock In {\em Annales Polonici Mathematici}, volume~1, pages 8--13, 1957.

\bibitem{PseudoLeja2012}
Leokadia Bia{\l}as-Cie{\.z} and Jean-Paul Calvi.
\newblock Pseudo {L}eja sequences.
\newblock {\em Annali di Matematica Pura ed Applicata}, 191:53--75, 2012.

\bibitem{CalviLeastSquares2008}
Jean-Paul Calvi and Norman Levenberg.
\newblock Uniform approximation by discrete least squares polynomials.
\newblock {\em Journal of Approximation Theory}, 152(1):82--100, 2008.

\bibitem{ReviewWAM2011}
Len Bos, Stefano De~Marchi, Alvise Sommariva, and Marco Vianello.
\newblock Weakly admissible meshes and discrete extremal sets.
\newblock {\em Numerical Mathematics: Theory, methods and applications},
  4(1):1--12, 2011.

\bibitem{RandomisedWAM2023}
Yiming Xu and Akil Narayan.
\newblock Randomized weakly admissible meshes.
\newblock {\em Journal of Approximation Theory}, 285:105835, 2023.

\bibitem{Ransford1995}
Thomas Ransford.
\newblock {\em Potential theory in the complex plane}.
\newblock Number~28. Cambridge university press, 1995.

\bibitem{Gaier1987lectures}
Dieter Gaier.
\newblock {\em Lectures on complex approximation}, volume 188.
\newblock Springer, 1987.

\bibitem{Totik2023}
Vilmos Totik.
\newblock The {L}ebesgue constants for {L}eja points are subexponential.
\newblock {\em Journal of Approximation Theory}, 287:105863, 2023.

\bibitem{Totik2010}
Rodney Taylor and Vilmos Totik.
\newblock Lebesgue constants for {L}eja points.
\newblock {\em IMA Journal of Numerical Analysis}, 30(2):462--486, 2010.

\bibitem{Lebesgue_Leja_Disk2013}
Moulay~Abdellah Chkifa.
\newblock On the {L}ebesgue constant of {L}eja sequences for the complex unit
  disk and of their real projection.
\newblock {\em Journal of Approximation Theory}, 166:176--200, 2013.

\bibitem{Explicit_Lebesgue_Interval2022}
Vladimir Andrievskii and Fedor Nazarov.
\newblock A simple upper bound for {L}ebesgue constants associated with {L}eja
  points on the real line.
\newblock {\em Journal of Approximation Theory}, 275:105699, 2022.

\bibitem{FastLeja1998}
J~Baglama, D~Calvetti, and L~Reichel.
\newblock Fast {L}eja points.
\newblock {\em Electron. Trans. Numer. Anal}, 7(124-140):119--120, 1998.

\bibitem{Bloom1992}
Thomas Bloom, Len Bos, C~Christensen, and Norman Levenberg.
\newblock Polynomial interpolation of holomorphic functions in $\mathbb{C}$ and
  $\mathbb{C}^n$.
\newblock {\em The Rocky Mountain Journal of Mathematics}, 22(2):441--470,
  1992.

\bibitem{NumericalCapacity2008}
W.~Dijkstra and M.E. Hochstenbach.
\newblock {\em Numerical approximation of the logarithmic capacity}.
\newblock CASA-report. Technische Universiteit Eindhoven, 2008.

\bibitem{NumericalCapacity2014}
Susanna Liesipohja.
\newblock {\em Numerical methods for computing logarithmic capacity}.
\newblock PhD thesis, M. Sc. thesis, University of Helsinki, 2014, 2014.

\bibitem{Tsuji1959}
Masatsugu Tsuji.
\newblock {\em Potential theory in modern function theory}.
\newblock Maruzen, 1959.

\bibitem{Kroo2011}
Andr{\'a}s Kro{\'o}.
\newblock On optimal polynomial meshes.
\newblock {\em Journal of Approximation Theory}, 163(9):1107--1124, 2011.

\bibitem{Pritsker1997}
Igor~E Pritsker.
\newblock Comparing norms of polynomials in one and several variables.
\newblock {\em Journal of Mathematical Analysis and Applications},
  216(2):685--695, 1997.

\bibitem{BookPolynomials1994}
Gradimir~V Milovanovic, Themistocles~M Rassias, and DS~Mitrinovic.
\newblock {\em Topics in polynomials: extremal problems, inequalities, zeros}.
\newblock World Scientific, 1994.

\bibitem{ReviewPolynomials1994}
Gradimir~V Milovanovic, Themistocles~M Rassias, and DS~Mitrinovic.
\newblock {\em Topics in polynomials: extremal problems, inequalities, zeros}.
\newblock World Scientific, 1994.

\bibitem{BernsteinMarkovReview2021}
Sergei Kalmykov, B{\'e}la Nagy, and Vilmos Totik.
\newblock Bernstein-and {M}arkov-type inequalities.
\newblock {\em arXiv preprint arXiv:2104.02348}, 2021.

\bibitem{MarkovPommerenke1959}
Ch~Pommerenke.
\newblock On the derivative of a polynomial.
\newblock {\em Michigan Math. J.}, 6(1):373--375, 1959.

\bibitem{BookMCMC1999}
Christian~P Robert, George Casella, and George Casella.
\newblock {\em Monte Carlo statistical methods}, volume~2.
\newblock Springer, 1999.

\bibitem{IMH2024}
Austin Brown and Galin~L Jones.
\newblock Exact convergence analysis for {M}etropolis--{H}astings independence
  samplers in {W}asserstein distances.
\newblock {\em Journal of Applied Probability}, 61(1):33--54, 2024.

\bibitem{Tierney1994}
Luke Tierney.
\newblock Markov chains for exploring posterior distributions.
\newblock {\em The Annals of Statistics}, pages 1701--1728, 1994.

\bibitem{Calvi_UnitDisk_2011}
Jean-Paul Calvi and Manh~Phung Van.
\newblock On the {L}ebesgue constant of {L}eja sequences for the unit disk and
  its applications to multivariate interpolation.
\newblock {\em Journal of Approximation Theory}, 163(5):608--622, 2011.

\end{thebibliography}

\end{document}